\documentclass[a4paper,11pt]{article}

\usepackage[english]{babel}
\usepackage{amsmath,amssymb,amsthm,mathrsfs,amsopn}
\usepackage{graphicx,color}

\newcommand{\weakto} {\rightharpoonup}                 
\newcommand{\weakstarto}{\stackrel {*} {\weakto}}      

\newcommand{\aac}{<\!\!<}
\newcommand{\defeq} {:=}
\newcommand{\mean}[1]{\,-\hskip-1.08em\int_{#1}} 
\newcommand{\meantext}[1]{\,-\hskip-0.88em\int_{#1}} 

\newcommand{\res}{\mathop{\hbox{\vrule height 7pt width .5pt depth 0pt
\vrule height .5pt width 6pt depth 0pt}}\nolimits} 

\newcommand{\R}{\mathbb{R}}

\newcommand{\Haus}[1]{{\mathscr H}^{#1}} 
\newcommand{\Leb}[1]{{\mathscr L}^{#1}} 

\newtheorem{theorem}{Theorem}[section]
\newtheorem{lemma}[theorem]{Lemma}
\newtheorem{proposition}[theorem]{Proposition}
\newtheorem{corollary}[theorem]{Corollary}
\theoremstyle{definition}

\theoremstyle{remark}
\newtheorem{remark}[theorem]{Remark}

\date{March 10, 2012}

\title{A nonautonomous chain rule in $W^{1,p}$ and $BV$}

\author{Luigi Ambrosio\footnote{Scuola Normale Superiore,
p.za dei Cavalieri 7, I-56126 Pisa, Italy, e--mail:
l.ambrosio@sns.it},
Graziano Crasta \footnote{Dipartimento di Matematica G.~Castelnuovo, Univ.\ di Roma La Sapienza, P.le A.~Moro 2, I-00185 Roma, Italy, e-mail: crasta@mat.uniroma1.it },
Virginia De Cicco\footnote{Dipartimento di Scienze di Base e Applicate per l'Ingegneria, Univ.\ di Roma La Sapienza, Via A. Scarpa 10, I-00185 Roma, Italy, e-mail: virginia.decicco@sbai.uniroma1.it },
Guido De Philippis\footnote{Scuola Normale Superiore,
p.za dei Cavalieri 7, I-56126 Pisa, Italy, e--mail:
guido.dephilippis@sns.it}}

\begin{document}
\maketitle
\section{Introduction}

The aim of this paper is to prove a generalization of the following
chain rule formula in $BV$ and in Sobolev spaces. Let
$F\colon\mathbb{R}^h \to \mathbb{R}$ be a $C^1$ function with
bounded gradient. It is well-known that, if $u\in BV_{\rm
loc}(\mathbb{R}^n;\mathbb{R}^h)$, then the composite function $v(x)
\defeq F(u(x))$ belongs to $BV_{\rm loc}(\mathbb{R}^n)$ and the
following chain rule formula holds in the sense of measures:
\begin{itemize}
\item[(i)] (diffuse part)
$\widetilde{D}v = \nabla F(\tilde{u}) \widetilde{D} u$;
\item[(ii)] (jump part)
$D^j v = [F(u^+) - F(u^-)] \nu_{J_u}\Haus{n-1}\res J_u$,
\end{itemize}
where $\widetilde{D} u$, $\tilde{u}$, $J_u$, $\nu_{J_u}$ and
$u^{\pm}$ denote respectively the diffuse part of the measure $Du$,
the precise representative of $u$, the jump set of $u$, the normal
vector to $J_u$ and the one-sided traces of $u$ (see
\cite[Thm.~3.96]{AFP}). The $C^1$ regularity of $F$ can be omitted
(requiring $F$ to be only Lipschitz continuous), but in this case
since the image of $u$ might be a low-dimensional set the gradient
$\nabla F$ appearing in (i) should be properly understood, see
\cite{ADM} and \cite{LM}; moreover, an analogous result holds true
also in the vectorial case $F\colon \mathbb{R}^h \to \mathbb{R}^p$.

In recent years this formula has been generalized in
order to deal with an explicit non-smooth dependence of $F$ from the
space variable $x$,
especially in view to applications to semicontinuity results for integral
functional (see \cite{DCFV1,DCFV2}) and to hyperbolic systems of conservation laws (see \cite{CDC}).

The result proved in this paper goes in this direction: given
$F\colon \mathbb{R}^n\times\mathbb{R}^h\to\mathbb{R}$, with $F(x,
\cdot)$ of class $C^1(\mathbb{R}^h)$ for almost every
$x\in\mathbb{R}^n$ and $F(\cdot, z)\in BV_{\rm loc}(\mathbb{R}^n)$
for all $z\in \mathbb{R}^h$, we establish the validity of a chain
rule formula for the composite function $v(x) = F(x, u(x))$ with
$u\in  BV_{\rm loc}(\mathbb{R}^n;\mathbb{R}^h)$. We believe that, in
our general setting, this formula can be a useful tool in the
analysis of the problems mentioned above, in cases where the
dependence from the space variables is of $BV$ type. In addition,
our result provides also a chain rule in the case when $u$ and
$F(\cdot,z)$ are Sobolev, see \eqref{eq:chainSobolev}.

The scalar case $h=1$ has been considered  in \cite{DCFV1} and
\cite{DCFV2} in the case of a dependence of $F$ with respect to $x$
respectively of Sobolev type and of $BV$ type.

The one-dimensional case $n=1$ has been studied in \cite{CDC}, where
a very explicit formula has been obtained at the price of some
additional assumptions (see Remark~\ref{compare} for a detailed
comparison with the set of assumptions of the present paper).

In this paper we prove a chain rule formula in the general case
$n\geq 1$ and $h\geq 1$ under some structural assumptions that now
we briefly describe. According to the classical case mentioned at
the beginning, the first additional assumption is a uniform bound on
$\nabla_z F(x,z)$ (see (H1) below). Concerning the $x$-derivative,
we need to require the existence of a Radon measure $\sigma$
bounding from above all measures $|D_x F(\cdot, z)|$, uniformly with
respect to $z\in\mathbb{R}^h$ (see assumption (H4) below). With
these two bounds we can prove that for any $u\in BV_{\rm loc}$ the
composite function $v(x) = F(x,u(x))$ belongs to $BV_{\rm loc}$ (see
Lemma~\ref{vBV} and Remark~\ref{remH4}). Moreover, we can show the
existence of a countably $\Haus{n-1}$-rectifiable set
$\mathcal{N}_F$, independent of $u$ and containing the jump set of
$F(\cdot,z)$ for every $z\in\mathbb{R}^h$, such that the jump set of
$v$ is contained in $\mathcal{N}_F\cup J_u$.

On the other hand, in order to prove the validity of the chain rule
formula we require that $F$ satisfies other structural assumptions
related to the uniform continuous dependence of $\nabla_z F$ and
$\widetilde{D}_x F$ with respect to $z$ (see assumptions (H2) and
(H3) below). All these assumptions are enough in order to prove the
following chain rule formula:
\begin{itemize}
\item[(i)] (diffuse part) $|Dv|\ll\sigma+|Du|$ and, for any Radon measure $\mu$ such that
$\sigma+|Du|\ll\mu$, it holds
\[
\frac{d\widetilde{D}v}{d\mu} =\frac{d\widetilde{D}_x F(\cdot,\tilde
u(x))}{d\mu} +\nabla_z\tilde F(x,\tilde u(x))\frac{d\widetilde{D}
u}{d\mu}.
\]
\item[(ii)] (jump part)
$D^j v=\big(F^+(x,u^+(x))-F^-(x,u^-(x)\big)\nu_{\mathcal N_F\cup
J_u}\Haus{n-1}\res \mathcal (\mathcal N_F\cup J_u)$
in the sense of measures.
\end{itemize}

In particular, when $u$ and $F(\cdot,z)$ belong to a Sobolev space, we obtain that the composite function $v$ belongs
to the same Sobolev  space and the following formula holds:
\begin{equation}\label{eq:chainSobolev1}
\nabla v(x)=\nabla_x F(x,u(x))+\nabla_z F(x,u(x))\nabla u(x) \qquad
\text{$\Leb{n}$-a.e. in $\R^n$}
\end{equation}
(see \cite{DCL} for a analogous result in the more general context of vector fields with $L^1$ divergence).

In order to prove (i) we use a blow-up argument as in the proof of
Theorem~3.96 in \cite{AFP}, which allows to treat at the same time
all higher dimensional cases $n\geq 1$ and $h\geq 1$ (in this
respect we recall that the approach based on convolutions does not
seem to work in the case $h>1$).

The explicit non-smooth dependence of $F$ with respect to $x$ gives
rise to several major technical difficulties. It is then crucial to
firstly investigate some fine properties of the function $F$ (see
Section~\ref{s:fine}). For example, we show that $\Haus{n-1}$-a.e.\
in $\mathcal{N}_F$ the one-sided traces $F^{\pm}(\cdot, z)$ exist
for all $z\in\mathbb{R}^h$ (see Proposition~\ref{convuniforme}).

\smallskip
\noindent {\bf Acknowledgement.} The first and fourth authors
acknowledge the support of the ERC ADG GeMeThNES.

\section{Notation and preliminary results}

In this section we recall our main notation and preliminary facts on
$BV$ functions. A general reference is Chapter 3 of \cite{AFP}, and
occasionally we will give more precise references therein.

We denote by $\Leb{n}$ the Lebesgue measure in $\R^n$ and by
$\Haus{k}$ the $k$-dimensional Hausdorff measure. The restriction of
$\Haus{k}$ to a set $A$ will be denoted by $\Haus{k}\res A$, so that
$\Haus{k}\res A(B)=\Haus{k}(A\cap B)$. By $\meantext{A}$ we mean
averaged integral on a set $A$. By Radon measure we mean a
nonnegative Borel measure finite on compact sets.

We say that $f\in L^1(\R^n)$ belongs to $BV(\R^n)$ if its derivative
in the sense of distributions is representable by a vector-valued
measure $Df=(D_1f,\ldots,D_nf)$ whose total variation $|Df|$ is
finite, i.e.
$$
\int_{\R^n}f\frac{\partial\phi}{\partial x_i} dx=-\int\phi\,D_if
\qquad\forall\phi\in C^\infty_c(\R^n),\,\,i=1,\ldots,n
$$
and $|Df|(\R^n)<\infty$. The $BV_{\rm loc}$ definition is analogous,
requiring that $|Df|$ is a Radon measure in $\R^n$.

\noindent {\bf Approximate continuity and jump points.} We say that
$x\in\R^n$ is an approximate continuity point of $f$ if, for some
$z\in\R$, it holds
$$
\lim_{r\downarrow 0}\mean{B_r(x)}|f(y)-z| dy=0.
$$
The number $z$ is uniquely determined at approximate continuity
points and denoted by $\tilde{f}(x)$, the so-called approximate
limit of $f$ at $x$. The complement of the set of approximate
continuity points, the so-called singular set of $f$, will be
denoted by $S_f$.

Analogously, we say that $x$ is a jump point of $f$, and we write
$x\in J_f$, if there exists a unit vector $\nu\in{\bf S}^{n-1}$ and
$f^+,\,f^-\in\R$ satisfying $f^+\neq f^-$ and
$$
\lim_{r\downarrow 0}\mean{B^\pm(x,r)}|f(y)-f^\pm| dy=0,
$$
where $B^\pm(x,r):=\{y\in B_r(x): \pm\langle y-x,\nu\rangle\geq 0\}$
are the two half balls determined by $\nu$. At points $x\in J_f$ the
triplet $(f^+,f^-,\nu)$ is uniquely determined up to a permutation
of $(f^+,f^-)$ and a change of sign of $\nu$; for this reason, with
a slight abuse of notation, we do not emphasize the $\nu$ dependence
of $f^\pm$ and $B^\pm(x,r)$. Since we impose $f^+\neq f^-$, it is
clear that $J_f\subset S_f$. Moreover
\begin{equation}\label{inclju}
J_f\subset\bigl\{x: \limsup_{r\downarrow
0}\frac{|Df|(B_r(x))}{\omega_{n-1}r^{n-1}}>0\bigr\}
\end{equation}
and
\begin{equation}\label{inclju2}
\Haus{n-1}\Big(\bigl\{x: \limsup_{r\downarrow
0}\frac{|Df|(B_r(x))}{\omega_{n-1}r^{n-1}}>0\bigr\}\setminus
J_f\Big)=0,
\end{equation}
see the proof of \cite[Lemma 3.75]{AFP}.

Recall that a set $\Sigma$ is said to be countably $\Haus{n-1}$
rectifiable if $\Haus{n-1}$-almost all of $\Sigma$ can be covered by
a sequence of $C^1$ hypersurfaces. For any $BV_{\rm loc}$ function
$f$, the set $S_f$ is countably $\Haus{n-1}$ rectifiable and
$\Haus{n-1}(S_f\setminus J_f)=0$.

\noindent \noindent {\bf Decomposition of the distributional
derivative.}

For any oriented and countably $\Haus{n-1}$-rectifiable
$\Sigma\subset\R^n$ we have
\begin{equation}\label{rappDjSigma}
Df\res \Sigma=(f^+-f^-)\nu_\Sigma\Haus{n-1}\res\Sigma.
\end{equation}

For any $f\in BV(\R^n)$, we can decompose $Df$ as the sum of a
diffuse part, that we shall denote $\widetilde{D}f$, and a jump part,
that we shall denote by $D^jf$. The diffuse part is characterized by
the property that $|\widetilde{D}f|(B)=0$ whenever $\Haus{n-1}(B)$ is
finite, while the jump part is concentrated on a set $\sigma$-finite
with respect to $\Haus{n-1}$. The diffuse part can be then split as
\[
\widetilde{D} f= D^a f+D^c f
\]
where $D^a f$ is the absolutely continuous part with respect to the
Lebesgue measure, while $D^c f$ is the so-called Cantor part. The
density of $Df$ with respect to $\Leb n$ can be represented as
follows
\begin{equation}\label{absolute}
D^a f= \nabla f \, d\Leb n,
\end{equation}
where $\nabla f$ is the approximate gradient of $f$, i.e. the only vector such that
\begin{equation}\label{appgrad}
\lim_{r \downarrow 0} \frac{1}{r^{n+1}} \int_{B_r(x)}
|f(y)-f(x)-\nabla f(x) \cdot (y-x)|=0\qquad\text{for almost every
$x$,}
\end{equation}
see \cite[Proposition 3.71 and Theorem 3.83]{AFP}.

 The jump part can be easily computed
by taking $\Sigma=J_f$ (or, equivalently, $S_f$) in
\eqref{rappDjSigma}, namely
\begin{equation}\label{rappDj}
D^jf=Df\res J_f=(f^+-f^-)\nu_{J_f}\Haus{n-1}\res J_f.
\end{equation}

All these concepts and results extend, mostly arguing component by
component, to vector-valued $BV$ functions, see \cite{AFP} for
details.

\section{The chain rule}

Let $F\colon\R^n\times\R^h\to\R$ be satisfying:
\begin{enumerate}
\item[(a)] $x\mapsto F(x,z)$ belongs to $BV_{\rm loc}(\R^n)$ for all
$z\in\R^h$;
\item[(b)] $z\mapsto F(x,z)$ is continuously differentiable in $\R^h$ for almost every $x\in\R^n$.
\end{enumerate}
We will use the notation $\nabla_z F(x,z)$ to denote the
(continuous) gradient of $z\mapsto F(x,z)$ and $D_x F(\cdot,z)$ to
denote the distributional gradient of $x\mapsto F(x,z)$. We will use
the notation $C_F$ to denote a Lebesgue negligible set of points
such that $F(x,\cdot)$ is $C^1$ for all $x\in\R^n\setminus C_F$.

We assume throughout this paper that $F$ satisfies, besides (a) and
(b), the following \emph{structural assumptions}:

\begin{enumerate}
\item[(H1)]For some constant $M$, $|\nabla_z F(x,z)|\le M$ for all
$x\in\R^n\setminus C_F$ and $z\in\R^h$.
\item[(H2)]For any compact set $H\subset\R^h$ there exists a modulus of continuity
$\tilde\omega_H$ independent of $x$ such that
\[
 |\nabla_z F(x,z)-\nabla_z F(x,z')|\leq \tilde \omega_H(|z-z'|)
 \]
for all $z,\,z'\in H$ and $x\in\R^n\setminus C_F$.
\item[(H3)]For any compact set $H\subset\R^h$ there exist a positive Radon measure
$\lambda_H$ and a modulus of continuity $\omega_H$ such that
\[
|\widetilde{D}_x F(\cdot,z)(A)-\widetilde{D}_x F(\cdot,z')(A)|\le
\omega_H(|z-z'|)\lambda_H(A)
\]
for all $z,\,z'\in H$ and $A\subset\R^n$ Borel.
\item[(H4)] The measure
\begin{equation}\label{defsigma}
\sigma\defeq \bigvee_{z\in\R^h} |D_x F(\cdot,z)|,
\end{equation}
(where $\bigvee$ denotes the least upper bound in the space of
nonnegative Borel measures) is finite on compact sets, i.e. it is a
Radon measure.
\end{enumerate}

{\color{blue} In connection with (H4), notice that already (H3) ensures that the supremum on bounded sets
of $\R^h$ of the diffuse parts is bounded. Hence, at least if we consider locally bounded functions $u$, the new requirement in
(H4) is for the jump parts, see also condition (A1) in Remark~\ref{compare}. Many variants of these assumptions are indeed possible,
for instance one could require in (H3)  a bounded and global modulus of continuity and require (H4) only for the jump parts.}

Still in connection with (H4), notice that an equivalent formulation of it
would be to require the existence of a Radon measure $\theta$
bounding from above all measures $|D_x F(\cdot,z)|$. Taking the
least possible $\theta$ has some advantages, as the following remark
shows.

\begin{remark}\label{proprietadisigma} {\rm The measure $\sigma$ in \eqref{defsigma}
vanishes on every $\Haus{n-1}$-negligible and on every purely
$(n-1)$-unrectifiable set (i.e. a set $B$ such that
$\Haus{n-1}(B\cap\Gamma)=0$ whenever $\Gamma$ is countably
$\Haus{n-1}$-rectifiable). Indeed, these properties are valid for
all the measures $|D_x F(\cdot,z)|$ thanks to the coarea formula,
see for instance \cite[Theorem~3.40]{AFP}.}
\end{remark}

We can now canonically build a countably $\Haus{n-1}$-rectifiable
set $B_\sigma$ containing all jump sets of $F(\cdot,z)$ as follows.
Indeed, we define
\begin{equation}\label{bsigma}
B_\sigma=\bigl\{x: \limsup_{r\downarrow 0}
\frac{\sigma(B_r(x))}{\omega_{n-1}r^{n-1}}
>0\bigr\}.
\end{equation}
Writing $B_\sigma$ as the union of the sets
\[
\bigl\{x\in B_k(0):\ \limsup_{r\downarrow  0}
\frac{\sigma(B_r(x))}{\omega_{n-1}r^{n-1}}
>\frac{1}{k}\bigr\}\qquad k\geq 1
\]
it is immediate to check that $B_\sigma$ is $\sigma$-finite with
respect to $\Haus{n-1}$. Now, according to \cite[Page~252]{Fed},
we can split $B_\sigma$ as a disjoint union $B_\sigma=L\cup U$ with
$L$ countably $\Haus{n-1}$-rectifiable and $U$ purely
$(n-1)$-unrectifiable. In order to prove the rectifiability of
$B_\sigma$, we are thus led to show that $\Haus{n-1}(U)=0$. To see
this notice that
\[
\Haus{n-1}\left(U\cap\left\{\limsup_{r\downarrow 0}
\frac{\sigma(B_r(x))}{\omega_{n-1}r^{n-1}} \ge \varepsilon\right\}\right)\le
\frac 1\varepsilon\sigma(U)=0
\]
thanks to Remark~\ref{proprietadisigma}. By \eqref{inclju} and the
inequality $\sigma\geq|D_x F(\cdot,z)|$ we know that $B_\sigma\supset
J_{F(\cdot,z)}$ for all $z\in\R^h$.

The main result of the paper is the following chain rule.

\begin{theorem}\label{chain rule}
Let $F$ be satisfying (a), (b), (H1)-(H2)-(H3)-(H4) above.
Then there exists a countably $\Haus{n-1}$-rectifiable set $\mathcal N_F$ such that,
for any function $u\in BV_{\rm loc}(\R^n;\R^h)$,
the function $v(x)\defeq F(x,u(x))$ belongs to $
BV_{\rm loc}(\R^n)$ and the following chain rule holds:
\begin{itemize}
\item[(i)] (diffuse part) $|Dv|\ll\sigma+|Du|$ and, for any Radon measure $\mu$
such that $\sigma+|Du|\aac\mu$, it holds
\begin{equation}\label{diffuse}
\frac{d\widetilde{D}v}{d\mu} =\frac{d\widetilde{D}_x F(\cdot,\tilde
u(x))}{d\mu} +\nabla_z\tilde F(x,\tilde u(x))\frac{d\widetilde{D}
u}{d\mu}\qquad\text{$\mu$-a.e. in $\R^n$.}
\end{equation}
\item[(ii)] (jump part)
$J_v\subset \mathcal N_F\cup J_u$
and, denoting by $u^\pm(x)$ and $F^\pm(x,z)$ the one-sided traces of
$u$ and $F(\cdot,z)$ induced by a suitable orientation of $\mathcal
N_F\cup J_u$, it holds
\begin{equation}\label{jump}
D^j v=\big(F^+(x,u^+(x))-F^-(x,u^-(x)\big)\nu_{\mathcal N_F\cup
J_u}\Haus{n-1}\res \mathcal (\mathcal N_F\cup J_u)
\end{equation}
in the sense of measures.
\end{itemize}
Moreover for a.e. $x$ the map $y\mapsto F(y,u(x))$ is
approximately differentiable at $x$ and
\begin{equation}\label{eq:appdif}
\nabla v(x)=\nabla_xF(x,u(x))+\nabla_zF(x, u(x))\nabla u(x) \qquad
\text{$\Leb{n}$-a.e. in $\R^n$\,.}
\end{equation}
\end{theorem}


Here (and in the sequel) the expression
\[
\frac{d\widetilde{D}_x F(\cdot,\tilde u(x))}{d\mu}
\]
means the pointwise density of the measure $\widetilde{D}_x F(\cdot,z)$
with respect to $\mu$, computed choosing $z=\tilde u(x)$ (notice
that the composition is Borel measurable thanks to the
Scorza-Dragoni Theorem and Lemma~\ref{serve} below). Analogously,
the expression $\tilde F(x,z)$ is well defined at points $x$ such
that $x\notin S_{F(\cdot,z)}$ and we will prove that,
$\nabla_z\tilde F(x,z)$ is well defined for all $z$ out of a
countably $\Haus{n-1}$-rectifiable set of points $x$.

\begin{remark}\label{misure}{\rm It is an easy exercise of measure theory to see that
\eqref{diffuse} holds for \emph{any} Radon measure $\mu$ such that
$\sigma+|Du|\aac\mu$ if and only if it holds for \emph{one} such
measure. For this reason, we shall prove the formula with
$\mu=\sigma+M|Du|$.}
\end{remark}
%



In the Sobolev framework the following chain rule holds (see \cite{DCL} for a similar result
relative to vector fields with controlled divergence).

\begin{corollary}\label{coroll}
Let $F$ be satisfying (b), (H1), (H2) above and the following conditions:
\begin{enumerate}
\item[(a)']
The function $x\mapsto F(x,z)$ belongs to $W^{1,1}_{loc}(\R^n)$ for all $z\in\R^h$;
\item[(H3)']For any compact set $H\subset\R^h$ there exist a positive function
$g_H\in L^1_{\rm loc}(\R^n)$ and a modulus of continuity $\omega_H$ such that
\[
|\nabla_x F(x,z)-\nabla_x F(x,z')|\le
\omega_H(|z-z'|)g_H(x)
\]
for all $z,\,z'\in H$ and  for a.e. $x\in\R^n$;
\item[(H4)']  There exists a positive function
$g\in L^1_{\rm loc}(\R^n)$ such that
\[
 |\nabla_x F(x,z)|\leq g(x)
\]
for all $z\in \R^h$ and  for a.e.\ $x\in\R^n$.
\end{enumerate}
Then for any function $u\in W^{1,1}_{\rm loc}(\R^n;\R^h)$
the function $v(x) \defeq F(x,u(x))$ belongs to
$W^{1,1}_{\rm loc}(\R^n)$ and
\begin{equation}\label{eq:chainSobolev}
\nabla v(x)=\nabla_x F(x,u(x))+\nabla_z F(x,u(x))\nabla u(x) \qquad
\text{$\Leb{n}$-a.e. in $\R^n$}.
\end{equation}
\end{corollary}


\begin{remark}[Locally bounded functions in domains $\Omega$]\label{localH1}
We stated the results for globally defined functions $u$, but
obviously it extends to the case when the domain is an open set
$\Omega$. In addition, if $u$ is locally bounded in $\Omega$, then
Theorem~\ref{chain rule} holds true replacing all the assumptions
with their local counterpart, we show for instance how (H1) and (H4)
should be modified:
\begin{itemize}
\item[(H1-loc)] For any pair of compact sets $K\subset\Omega$ and $H\subset\R^h$
there exists a constant $M_{K,H}$ such that $|\nabla_z F(x,z)| \leq
M_{K,H}$ for every $x\in K\setminus C_F$ and $z\in H$.
\item[(H4-loc)] For any compact set $H\subset\R^h$ the measure
\begin{equation}\label{defsigmaH}
\sigma_H\defeq \bigvee_{z\in H} |D_x F(\cdot,z)|
\end{equation}
is a Radon measure in $\Omega$.
\end{itemize}
In this local case we can define $\mathcal{N}_F \defeq \cup_{j\geq
1} \mathcal{N}_j$, where
\[
\mathcal{N}_j \defeq
\bigl\{x: \limsup_{r\downarrow 0}
\frac{\sigma_j(B_r(x))}{\omega_{n-1}r^{n-1}}
>0\bigr\}\,,\qquad
\sigma_j \defeq \bigvee_{z\in B_j(0)} |D_x F(\cdot,z)|\,,
\]
and, in the any subdomain $\Omega'\Subset\Omega$ where $|u|\leq j$,
the measure $\sigma$ in Theorem~\ref{chain rule}(i) should be
replaced by $\sigma_j$.
\end{remark}

\begin{remark}\label{compare}
In \cite{CDC} the one-dimensional analogous (i.e.\ for $n=1$) of
Theorem~\ref{chain rule} has been proved under the following structural assumptions:
\begin{itemize}
\item[(A1)] There exists a countable set $\mathcal{N}_F$
containing $S_{F(\cdot, z)}$ for every $z\in\R^h$.
\item[(A2)] For every compact set $H\subset\R^h$ there exists a Radon measure
$\lambda_H$ such that
\[
|D_x F(\cdot, z)(A) - D_x F(\cdot, z')(A)| \leq |z-z'| \lambda_H(A)
\]
for all $z,z'\in H$ and every Borel set $A\subset\R$.
\item[(A3)] For every compact set $H\subset\R^h$ there exists a constant $M_H$
such that $|\nabla_z F(x,z)| \leq M_H$ for every $x\in\R\setminus\mathcal{N}_F$
and every $z\in H$.
\item[(A4)] $x\mapsto \nabla_z F(x,z)$ belongs to $BV(\R)$ for every $z\in\R^h$.
\item[(A5)] There exists a positive finite Cantor measure $\lambda$
(i.e., a measure whose diffuse part is orthogonal to the Lebesgue measure)
such that $D^c_x F(\cdot, z) \ll \lambda$ for every $z\in\R^h$.
\end{itemize}
It is apparent that the new assumption (H4) allows to drop both
assumptions (A1) (see the construction of the set $B_{\sigma}$ in
\eqref{bsigma}) and (A5). Moreover, assumption (H3) involves only
the diffuse part of the measure $|D_x F(\cdot, z)|$, and so it is
weaker than (A2). The other assumptions are almost equivalent (even
the ``local'' assumption (A3) is equivalent to (H1) since a $BV$
function of the real line is locally bounded).

\end{remark}


We first prove that under assumptions (H1) and (H4),
the composite function $v$ belongs to $BV_{\rm loc}(\R^n)$.

\begin{lemma}\label{vBV}
The function $v$ belongs to $BV_{\rm loc}(\R^n)$ and $|D v|\leq\sigma+M|Du|$.
\end{lemma}
\begin{proof} It is obvious that $v\in L_{\rm loc}^1(\R^n)$, so we only have to check that
$|D v|(K)<\infty$ for any compact set $K$. In order to do this
choose a standard family of mollifiers $\varrho_\varepsilon$ and
define
\[
F_\varepsilon(x,z)=F(\cdot,z)\ast\varrho_\epsilon(x)=
\int\varrho_\varepsilon(x-x')F(x',z)
dx'
\]
which is a $C^1$ function satisfying $\sup_{(x,z)} |\nabla_z
F_\varepsilon(x,z)|\le M$ and set
$v_\varepsilon(x)=F_\varepsilon(x,u(x))$. If $u\in C^1$ the standard
chain rule and the inequality
$$
\sup_z|\nabla_xF_\varepsilon(x,z)|\leq\sigma\ast\varrho_\epsilon(x)
$$
give, for any compact set $K$,
\begin{eqnarray*}
\int_K|\nabla v_\epsilon|\,dx&\leq&\int_K|\nabla_x
F_\varepsilon(x,u(x))|\,dx+\int_K|\nabla_z
F_\varepsilon(x,u(x))||\nabla u(x)|\,dx
\\&\leq&\int_K\sigma\ast\varrho_\epsilon(x)\,dx+M\int_K|\nabla u|\,dx.
\end{eqnarray*}
By approximation the same inequality holds, now with $|Du|(K)$ in
place of $\int_K|\nabla u|\,dx$, if $u\in BV_{\rm loc}$. Eventually
we can use the arbitrariness of $K$ to get
$|Dv_\varepsilon|\leq\sigma\ast\varrho_\epsilon\Leb{n}+M|Du|$.

The only thing to check is that $v_\varepsilon\to v$ in $L^1_{\rm
loc}(\R^n)$, to do this it suffices to check the pointwise
convergence since
\[
|v_\varepsilon(x)|\le
|F(\cdot,0)|\ast\varrho_\varepsilon(x)+M|u|\ast\varrho_\varepsilon(x)
\]
and the right-hand side is convergent in $L^1_{\rm loc}$. We know
that for fixed $z\in\R^h$ it holds that
\[
\lim_{\varepsilon \to 0} F_\varepsilon(x,z)= F(x,z)
\]
for every $x\in A_z$ where $A_z$ is the set of Lebesgue points of
$x\to F(x,z)$. To prove the almost everywhere convergence of
$F_\varepsilon (x,u(x))\to F(x,u(x))$ it is thus enough to show that
\[
\Leb{n}\Big(\bigcup_{z\in \R^h} (\R^n\setminus A_z)\Big)=0.
\]
To see this let us choose a countable dense set $U\subset \R^h$; we
claim that
\[
\bigcap_{z\in U} A_z=\bigcap_{z\in \R^h} A_z.
\]
Indeed, if $x\in\bigcap\limits_{z\in U} A_z$, $z\in\R^h$ and $z_k
\in U$ converges to $z$, then for all $0<r<R$ it holds
\[
\begin{split}
&\Big|\mean{B_R(x)}F(y,z)\,dy-\mean{B_r(x)}F(y,z)\,dy\Big|
\le \mean{B_R(x)}|F(y,z)-F(y,z_k)|\,dy\\
&+\Big|\mean{B_R(x)}F(y,z_k)\,dy-\mean{B_r(x)}F(y,z_k)\,dy\Big|
+ \mean{B_r(x)}|F(y,z)-F(y,z_k)|\,dy\\
&\le
2M|z_k-z|+\Big|\mean{B_R(x)}F(y,z_k)\,dy-\mean{B_r(x)}F(y,z_k)\,dy\Big|.
\end{split}
\]
This proves that the averages $\meantext{B_r(x)} F(y,z)\, dy$ are Cauchy
as $r\downarrow 0$ for every $x\in\bigcap\limits_{z\in U} A_z$ and
$z\in \R^h$. Denoting by $\tilde{F}$ the limit, in the same way one
can prove that
\[
\mean{B_r(x)}|F(y,z)-\tilde F(x,z)|\,dy \to 0.
\]
The claim now immediately follows.
\end{proof}

\begin{remark}\label{remH4}
Even in the Sobolev case the assumption (H1) is needed in order to
have $v\in W^{1,1}_{\rm loc}$. For instance, if $F(x,z) = f(x) z$,
with $f\in W^{1,1}_{\rm loc}(\mathbb{R}^n)$ not locally bounded, the
composite function $v(x) = f(x) u(x)$ may not be locally integrable,
unless $u$ is locally bounded. But, even in that case, the term
$f\nabla u$ might be not locally integrable.
\end{remark}

\begin{lemma}\label{serve} Any Radon measure $\mu$ in $\R^n$
is concentrated on a Borel set $A$ with the following property:
\begin{equation}\label{eq:densitiesmeaning}
\frac{d\widetilde{D}_x F(\cdot,z)}{d \mu}(y) =\lim_{r\downarrow
0}\frac{\widetilde{D}_x F(\cdot,z)(B_r(y))}{\mu(B_r(y))}
\end{equation}
exists for every $y\in A$, $z\in\R^h$ and is Borel-measurable in $y$
and continuous in $z$, more precisely it holds
\[
\Bigg |\frac{d\widetilde{D}_x F(\cdot,z_1)}{d \mu}(y)-\frac{d\widetilde{D}_x
F(\cdot,z_2)}{d \mu}(y)\Bigg|\le \omega_{H} (|z_1-z_2|)\frac{d
\lambda_H}{d \mu}(y)
\]
for any $z_1,\,z_2 \in H$, $H\subset\R^h$ compact.
\end{lemma}
\begin{proof} Let us fix a compact $H\subset\R^h$, we will show that statement of the lemma
holds for any $z\in H$ and $y\in A_H$ with $\mu(\R^n\setminus
A_H)=0$. Then, writing $\R^h$ as a countable union of compact sets
$H_k$, we obtain the thesis with $A\defeq\bigcap_k A_{H_k}$. In the
sequel $H$ will be fixed and when referring to the measure
$\lambda_H$ and the modulus $\omega_H$ appearing in (H3) we will drop the subscript.

Let $U\subset H$ be a countable dense set, then there exists a set
$A$ with $\mu(\R^n\setminus A)=0 $ such that for every $z\in U$ and
every $y\in A$ the following limit exists
\[
f(y,z):=\lim_{r\downarrow 0}\frac{\widetilde{D}_x
F(\cdot,z)(B_r(y))}{\mu(B_r(y))}.
\]
Moreover, possibly removing from $A$ a $\mu$-negligible set and using
Besicovitch differentiation theorem, we can assume that for every
$y\in A$ there exists the limit $\lim_r \lambda(B_r(y))/\mu(B_r(y))$
and coincides with a fixed version of the Radon-Nikod\'ym derivative
$d \lambda/d \mu$. We claim that any $z\in H$ has the required property:
indeed, for any such point $y$ we have
\[
|f(y,z_1)-f(y,z_2)|\le \omega(|z_1-z_2|)\frac{d\lambda}{d\mu}(y)
\]
for every $z_1,\,z_2\in U$. Let choose now $z\in H$ and $z_k\in U$,
$z_k\to z$, and define $f(y,z)=\lim_k f(y,z_k)$, which exists and
does not depend on $(z_k)$. Then
\[
\begin{split}
&\left|\frac{\widetilde {D}_x F(\cdot,z)( B_r(y))}{\mu(B_r(y))}-f(y,z)\right|\\
&\leq\left|\frac{\widetilde {D}_x F(\cdot,z_k)(
B_r(y))}{\mu(B_r(y))}-f(y,z_k)\right|+\omega(|z_k-z|)\frac{\lambda(B_r(y))}{\mu(B_r(y))}
+|f(y,z_k)-f(y,z)|.
\end{split}
\]
Passing to the limit first as $r\downarrow 0$ and then as
$k\to\infty$ we obtain the thesis. Borel measurability in $y$ and continuity in $z$ easily follow.
\end{proof}

\section{Fine properties of $F$}\label{s:fine}

The proof of the following lemma, concerning differentiability of
integrals depending on a parameter, is a direct consequence of the
dominated convergence theorem.

\begin{lemma}\label{medie} If $A\subset\R^n$ is a bounded measurable set, the function
\[
z\mapsto m_A(z)\defeq \mean{A}F(x,z) dx
\]
is continuously differentiable in $\R^h$ and with bounded gradient
given by
\[
\nabla_z m_A(z)=\mean A \nabla_z F(x,z)dx.
\]
In particular $|m_A(z)-m_A(z')|\le M|z-z'|$ and, if $\tilde\omega_H$ is
as in (H2), then
\begin{equation}\label{medie1}
|\nabla_z m_A(z)-\nabla_z m_A(z')|\le\tilde\omega_H(|z-z'|)
\qquad\text{for all $z,\,z'\in H$.}
\end{equation}
\end{lemma}

\begin{proposition}\label{convuniforme} The following two statements
hold:
\begin{enumerate}
\item[(i)] There exists a $\Haus{n-1}$-negligible set $A$ such that,
defining $B=B_\sigma\cup A$, for all $x\in\R^n\setminus B$ the
function $F(\cdot,z)$ is approximately continuous at $x$ for every
$z\in\R^h$ and the function $z\mapsto\tilde F(x,z)$ is $C^1$ with
bounded derivative.
\item[(ii)] Let $\Sigma$ be a countably $\Haus{n-1}$-rectifiable set oriented by
$\nu_\Sigma$. Then, for $\Haus{n-1}$-a.e. $x\in\Sigma$ the one-sided
limits $F^+(x,z)$ and $F^-(x,z)$ defined by
\[
\lim_{r\downarrow 0} \mean{B^\pm_r(x)}|F(y,z)-F^{\pm}(x,z)|dy=0
\]
exist for all $z\in\R^h$ and $z\mapsto F^\pm(x,z)$ is $C^1$ with
bounded derivative.
\end{enumerate}
\end{proposition}
\begin{proof} (i) Choose a countable dense set $U$ in $\R^h$
and $$ A:=\bigcup_{z\in U}S_{F(\cdot,z)}\setminus J_{F(\cdot,z)}.$$
This set is clearly $\Haus{n-1}$-negligible and, since $B_\sigma$
contains all jump sets of $F(\cdot,z)$, $B=A\cup B_\sigma$ contains
all sets $S_{F(\cdot,z)}$, $z\in U$, so that all points in
$\R^n\setminus B$ are approximate continuity points of $F(\cdot,z)$,
$z\in U$. Using assumption (H1) and a density argument as in
Lemma~\ref{serve} we obtain that all functions $F(\cdot,z)$ have
approximate limits at all points in $\R^n\setminus B$. Since
Lemma~\ref{medie} gives
$$
\nabla_z \mean{B_r(x)}F(y,z) dy=\mean{B_r(x)}\nabla_z F(y,z) dy
$$
we can pass to the limit as $r\downarrow 0$ and use \eqref{medie1}
to obtain that $\nabla_z\tilde{F}(x,z)$ exists and is continuous.

\noindent (ii) Arguing as in the proof of (i), we can find a
$\Haus{n-1}$-negligible set $A\subset\Sigma$ such that for every
$x\in\Sigma\setminus A$ the limits
\[
 F^{\pm}(x,z)\defeq\lim_{r\downarrow 0} \mean{B^\pm_r(x)} F(y,z)dy
 \]
exist for every $z\in\R^h$ and
\[
\lim_{r\downarrow 0}\mean{B^\pm(x)}|F(y,z)-F^\pm(x,z)|dy=0
\qquad\forall z\in\R^h.
\]
In addition, $z\mapsto F^\pm(x,z) $ is continuously differentiable
in $\R^h$.
\end{proof}

Thanks to the previous proposition we know that the Borel map
$x\mapsto\nabla_z\tilde F(x,z)$ is well defined for every
$x\in\R^n\setminus B$ and $z\in\R^h$. Since $B$ is $\sigma$-finite
with respect to $\Haus{n-1}$, $\nabla_z \tilde{F}(x,\tilde{f}(x))$
is well defined for $|\widetilde{D}f|$-almost every point $x$ for every
$BV$ function $f$.

In the next lemma we provide a more precise expression for the
derivative of $z\mapsto\tilde{F}(x,z)$.

\begin{lemma}\label{Convgrad} Let $x\notin\cup_z S_{F(\cdot,z)}$.
Then, for all $z\in\R^h$ the function $y\mapsto\nabla_z F(y,z)$ is
approximately continuous at $x$ and
\begin{equation}\label{m1}
\nabla_z \tilde{F}(x,z)=\widetilde{\nabla_z F}(\cdot,z)(x),
\end{equation}
where $\widetilde{\nabla_z F}(\cdot,z)(x)$ is the approximate limit
at $x$ of $y\mapsto\nabla_z F(y,z)$. In particular, for all $z\in\R^h$
the functions
$$
G^r(y):=\nabla_z F(x+ry,z)
$$
weak$^*$ converge in $L^\infty(B_1(0))$ to $\nabla_z\tilde F(x,z)$
as $r\downarrow 0$.\footnote{{\color{blue} Quest'ultimo fatto e' una conseguenza diretta di
quanto detto in \eqref{m1}}}
\end{lemma}
\begin{proof} Let $x$ a point where
\begin{equation}\label{m}
\mean{B_r(x)} |F(y,z)-F(x,z)| dz \to 0
\end{equation}
for every $z$. Fix a direction $e\in S^{n-1}$, $z_0\in\R^h$ and
consider for $h\in (0,1]$
\[
g_r(h):=\mean{B_r(x)}\bigl|\frac{F(y,z_0+he)-F(y,z_0)}{h}-\langle\nabla_z F(y,z_0),e\rangle\bigr| dy.
\]
We also consider a sequence $(h_i)\downarrow 0$ such that $h_i^{-1}(\tilde{F}(x,z_0+h_i)-\tilde{F}(x,z_0))$ converges to
$\xi$, and prove that the approximate limit of $\langle\nabla_zF(\cdot,z_0),e\rangle$ at $x$ equals $\xi$. Indeed,
by the mean value theorem and hypothesis (H2) one has $g_r(h)\to 0$ as $h\to 0$ uniformly for $r\in (0,1)$, therefore
$$
\limsup_{r\downarrow 0}\mean{B_r(x)}\bigl|\langle\nabla_z F(y,z_0),e\rangle-\xi| dy
$$
can be estimated from above with
$$
\limsup_{i\to\infty}\limsup_{r\downarrow 0}\mean{B_r(x)}
\biggl|\frac{F(y,z_0+h_ie)-F(y,z_0)}{h_i}-\frac{\tilde{F}(x,z_0+h_ie)-\tilde{F}(x,z_0)}{h_i}\biggr| dx.
$$
The latter is equal to 0 because of \eqref{m}. This proves the existence of the approximate limit.

In order to prove \eqref{m1} suffices to pass to the limit as $r\downarrow 0$ into the identity
$$
\nabla_z\mean{B_r(x)}F(y,z) dy=\mean{B_r(x)}\nabla_z F(y,z) dy.
$$
\end{proof}

In case $\mu = \Leb n$, Lemma~\ref{serve} can be refined in the
following way:
\begin{lemma}\label{lemma ac}
There exists a Lebesgue negligible set $E$ such that for any $x\in E^c$ and every $z \in \R^h$ it holds
\begin{equation}\label{ac}
\nabla_x F(x,z) =\lim_{r \downarrow 0}\frac{
D_x F(\cdot,z)(B_r(x))}{\Leb n(B_r(x))},
\end{equation}
where $\nabla _x F(x,z)$ is the approximate gradient at $x$ of
$y\mapsto F(y,z)$. Moreover $\nabla_x F(x,z)$ is Borel measurable in
$x$ and continuous in $z$.
\end{lemma}
\begin{proof}
By an easy approximation argument we can assume that $z$ varies in a
compact set $H$. Thanks to Lemma~\ref{serve}, we know that the limit
in \eqref{ac} exists for every $z\in \R^h$ and $x \in A$, where
$\Leb n(A^c)=0$. Moreover, if we call such limit $L(x,z)$ we have
for any $x\in A$, $z_1,\, z_2 \in H$ it holds
\begin{equation}\label{1}
|L(x,z_1)-L(x,z_2)|\le \omega_H(|z_1-z_2|)\frac{d\lambda_H}{d\Leb n
}(x).
\end{equation}
Possibly removing a negligible set we can also assume that for every $x$ in $A$,
$F(\cdot,z)$ is approximately continuous  for any $z$ (see
Lemma~\ref{convuniforme}),  $d\lambda_H/d\Leb n$ exists and that (since clearly $\sigma \res B_\sigma \perp \Leb n$)
\begin{equation}\label{sigmabsigma}
\lim_{r \downarrow 0} \frac{\sigma\res B_\sigma (B_r(x))}{r^n}=0.
\end{equation}
Let us now choose a countable dense set
$U \subset H$; thanks to~\cite[Theorem 3.83]{AFP} we know
that there exists a Borel set $\tilde B$ with Lebesgue negligible
complement such that for any $x\in \tilde B$ and $z\in U$
\begin{equation}\label{2}
\nabla_x F(x,z) =L(x,z)=\lim_{r \downarrow 0}\frac{\widetilde{D}_xF(\cdot,z)(B_r(x))}{\Leb n(B_r(x))}
\end{equation}
and
\[
\lim_{r \downarrow 0} \frac{1}{r^{n+1}} \int_{B_r(x)} 
|F(y,z)-F(x,z)-\nabla_x F(x,z)\cdot (y-x)|\, dy=0.
\]
Choosing $x\in E:=\tilde B\cap A$, $z\in \R^h$ and a sequence
$(z_k)\subset U$ converging to $z$, we have
\[
\begin{split}
& \frac{1}{r^{n+1}} \int_{B_r(x)} |F(y,z)-\tilde F(x,z)-L(x,z)\cdot (y-x)|\,dy\\
&\le \frac{1}{r^{n+1}} \int_{B_r(x)} |F(y,z_k)-\tilde F(x,z_k)-\nabla_x F(x,z_k)\cdot (y-x)|\,dy\\
&+\frac{1}{r^{n+1}} \int_{B_r(x)} |F(y,z)-F(y,z_k)-\tilde F(x,z)+\tilde F(x,z_k)-(\nabla_x F(x,z_k)-L(x,z))\cdot (y-x)|\,dy\\
&\le o_r (1)+\sup_{t\in (0,1)} \frac{|D_x F(\cdot,z)-D_x F(\cdot, z_k)|(B_{tr}(x))}{(tr)^n}+\omega_n |L(x,z)-\nabla_{x} F(x,z_k)|\\
& \le o_r (1)+\sup_{t\in (0,1)}\Bigg( \frac{|D_x F(\cdot,z)-D_x F(\cdot, z_k)|(B_{tr}(x)\cap (\R^n\setminus B_\sigma))}{(tr)^n}+\frac{2\sigma\res B_\sigma (B_{tr}(x))}{(tr)^n}\Bigg)\\
&+\omega_n |L(x,z)-\nabla_{x} F(x,z_k)|\\
&\le o_r (1)+\sup_{t\in (0,1)} \frac{|\widetilde{D}_x F(\cdot,z)-\widetilde{D}_x F(\cdot, z_k)|(B_{tr}(x))}{(tr)^n}+\omega_n |L(x,z)-\nabla_{x} F(x,z_k)|\\
&\le o_r (1)+\omega_H(|z-z_k|)\sup_{t\in (0,1)}\frac{\lambda_H
(B_{tr}(x))}{(tr)^n}+\omega_n|L(x,z)-\nabla_{x} F(x,z_k)|\\
&\le o_r (1)+\omega_H(|z-z_k|)\sup_{t\in (0,1)}\frac{\lambda_H
(B_{tr}(x))}{(tr)^n}+\omega_n\omega_H(|z-z_k|)\frac{d\lambda_H}{d\Leb n
}(x)\,,
\end{split}
\]
where in the second inequality we applied \cite[Lemma 3.81]{AFP} to
the function $F(\cdot,z)-F(\cdot,z_k)$, in the fourth one the fact that $D_x F(\cdot, z)\res (\R^n \setminus B_\sigma)\le \widetilde{D}_x F(\cdot, z) $ for any $z$ and equation \eqref{sigmabsigma}, in the last but one hypothesis (H3), and finally in the last one \eqref{1} and \eqref{2} (recall that $z_k\in U$). 
Passing to the limit first in $r$ and then in $k$ we obtain the thesis.
\end{proof}

\section{Proof of the chain rule}

In this section we prove Theorem~\ref{chain rule}.

Define $\mu=\sigma+M|Du|$ and
\[
J:=\bigl\{x: \limsup_{r\downarrow 0}
\frac{\mu(B_r(x))}{\omega_{n-1}r^{n-1}}>0\bigr\}.
\]
Notice that $J$ is $\sigma$-finite with respect to $\Haus{n-1}$ and
that
\[
\bigl\{x: \limsup_{r \downarrow 0}
\frac{|Du|(B_r(x))}{\omega_{n-1}r^{n-1}}>0\bigr\}\cup B_\sigma= J,
\]
where $B_\sigma$ is defined in \eqref{bsigma}. By \eqref{inclju} and
\eqref{inclju2} we also get
\begin{equation}\label{deco}
J\supset B_\sigma \cup
J_u\qquad\text{and}\qquad\Haus{n-1}\bigl(J\setminus (B_\sigma\cup
J_u)\bigr)=0.
\end{equation}

By Proposition~\ref{convuniforme} and \eqref{inclju} we can add to
$J$ a $\Haus{n-1}$-negligible set $A$ to obtain a set $\tilde
J=J\cup A$ satisfying
\[
\begin{split}
&\lim_{r \downarrow 0} \mean{B_r(x)}|F(y,z)-\tilde F(x,z)|dy=0\quad\forall z\in\R^h, \\
&\lim_{r \downarrow 0} \mean{B_r(x)}|u(y)-\tilde u(x)|dy=0
\end{split}
\]
for all $x\in\R^n\setminus\tilde J$. We claim that $v$ is
approximately continuous at all points in $\R^n\setminus\tilde J$
with precise representative $ \tilde v(x)=\tilde F(x,\tilde u(x))$.
To see this, compute
\[
\begin{split}
&\limsup_{r\downarrow 0} \mean{B_r(x)}|F(y,u(y))-\tilde F(x,\tilde u(x))|dy\\
&\le \limsup_{r\downarrow 0} \mean{B_r(x)}|F(y,u(y))-F(y,\tilde
u(x))|dy
+\limsup_{r\downarrow 0} \mean{B_r(x)}|F(y,\tilde u(x))-\tilde F(x,\tilde u(x))|dy\\
&\le\limsup_{r\downarrow 0} M\mean{B_r(x)}|u(y)-\tilde
u(x)|dy+\limsup_{r \downarrow 0} \mean{B_r(x)}|F(y,\tilde
u(x))-\tilde F(x,\tilde u(x))|dy=0.
\end{split}
\]
In particular it holds that
\begin{equation}\label{ritardo}
\begin{split}
&Dv\res (\R^n\setminus\tilde J)=Dv\res (\R^n\setminus J)=\widetilde{D}v,\\
&Du\res (\R^n\setminus\tilde J)=Du\res (\R^n\setminus J)=\widetilde{D}u,\\
&D_xF(\cdot,z)\res (\R^n\setminus\tilde J)=D_xF(\cdot,z)\res
(\R^n\setminus J)=\widetilde {D}_xF(\cdot,z)\quad\forall z\in\R^h.
\end{split}
\end{equation}
Since $D v\aac\mu$, in order to characterize $\widetilde{D}v$ it
suffices to show that for $\mu$-almost every
$x_0\in\R^n\setminus\tilde J$ it holds:
\begin{equation}\label{mu}
\frac{d\widetilde{D}v}{d\mu}(x_0)=\frac{d\widetilde{D}_xF(\cdot,\tilde{u}(x_0))}{d\mu}(x_0)+\nabla_z\tilde F(x_0,\tilde
u(x_0))\frac{d\widetilde{D}u}{d\mu}(x_0).
\end{equation}
Choose a point $x_0\in \R^n\setminus\tilde J$ such that
(understanding densities in the pointwise sense as in
\eqref{eq:densitiesmeaning}):
\begin{enumerate}
\item[(i)]there exists $d\widetilde{D}u/d\mu$ at $x_0$,
\item[(ii)]there exists $d\widetilde{D}v/d\mu$ at $x_0$,
\item[(iii)]there exists $d\widetilde{D}_xF(\cdot,z)/d\mu$ at $x_0$ for every
$z\in\R^h$,
\item[(iv)]there exists $d\lambda_H/d\mu$ at $x_0$, with $H$ compact neighborhood of $\tilde u(x_0)$,
\item[(v)] $x_0 $ is a point of density 1 for $\R^n \setminus \tilde J$ with respect to $\mu$,
\item[(vi)] ${\rm Tan}(\mu,x_0)$ contains a nonzero measure, i.e.
there exist $r_h\downarrow 0$ such that the measures
$\mu_{x_0,r_h}/\mu(B_{r_h}(x_0))$ (with $\mu_{x_0,r}(B)=\mu(x_0+rB)$)
weakly converge, in the duality with $C_c(B_1(0))$, to $\nu\neq 0$.
\end{enumerate}
Notice that $\mu$-almost every $x_0\in\R^n\setminus\tilde J$
satisfies the previous assumptions. Indeed, (iii) is satisfied
thanks to Lemma~\ref{serve}, while (vi) follows from \cite[Proposition 2.42]{AFP}.

Define for $y\in B_1(0)$
\[
u^r(y)\defeq
\frac{u(x_0+ry)-m_r(u)}{\mu(B_r(x_0))/r^{n-1}}\qquad\text{and}\qquad
v^r(y)\defeq \frac{v(x_0+ry)-m^F_r (m_r(u))}{\mu(B_r(x_0))/r^{n-1}}
\]
where
\[
 m_r(u)=\mean{B_r(x_0)} u(x)dx\qquad\text{and}\qquad
 m_r^F(z)=\mean{B_r(x_0)}F(x,z)dx.
\]

We claim that $v^r$ is
relatively compact in $L^1(B_1(0))$.
Namely, let us relate more precisely
$v^r$ to $u^r$, writing
\begin{equation}\label{alica1}
\begin{split}
v^r(y)&=\frac {F(x_0+ry,u(x_0+ry))-m^F_r(m_r(u))} {\mu(B_r(x_0))/r^{n-1}}\\
&=\frac 1 {\mu(B_r(x_0))/r^{n-1}}\Big\{F(x_0+ry,u(x_0+ry))-F(x_0+ry,m_r(u)))\\
&+F(x_0+ry,m_r(u))-m^F_r(m_r(u))\Big\}\\
&=\frac 1 {\mu(B_r(x_0))/r^{n-1}}\Big\{\nabla_z F(x_0+ry,m_r(u))(u(x_0+ry)-m_r(u))+R(y)\\
&+F(x_0+ry,m_r(u))-m^F_r(m_r(u))\Big\}\\
&=\nabla_z
F(x_0+ry,m_r(u))\frac{u(x_0+ry)-m_r(u)}{\mu(B_r(x_0))/r^{n-1}}
+\frac{F(x_0+ry,m_r(u))-m^F_r(m_r(u))}{\mu(B_r(x_0))/r^{n-1}}+\overline
R(y)
\end{split}
\end{equation}
with
\[
R(y)=\bigl(u(x_0+ry)-m_r(u)\bigr)\Big(\int_0^1 \nabla_z
F\big(x_0+ry,m_r(u)+t(u(x_0+ry)-m_r(u)\big)dt-\nabla_z
F(x_0+ry,m_r(u))\Big)
\]
and $\overline R=R r^{n-1}/\mu(B_r(x_0))$. By Poincar\'e inequality
\[
u^r(y)\defeq\frac{u(x_0+ry)-m_r(u)}{\mu(B_r(x_0))/r^{n-1}}
\]
is bounded in $BV(B_1(0))$ and therefore relatively compact in the
strong topology of $L^1(B_1(0))$. In addition
\begin{equation}\label{alica2}
\nabla_z F(x_0+ry,m_r(u))\weakstarto \nabla_z \tilde F(x_0,\tilde
u(x_0)) \quad\text{in $L^\infty(B_1(0))$}
\end{equation}
thanks to (H2) and Lemma~\ref{Convgrad}.

The second term in \eqref{alica1}, namely
\[
F^r(y)\defeq\frac{F(x_0+ry,m_r(u))-m^F_r(m_r(u))}{\mu(B_r(x_0))/r^{n-1}}
\]
satisfies
\[
|D_y F^r|(B_1(0))\le \frac{|D_x
F(\cdot,m_r(u))|(B_r(x_0))}{\mu(B_r(x_0))}\le
\frac{\sigma(B_r(x_0))}{\mu(B_r(x_0))}\le 1
\]
and
\[
\int_{B_1(0)} F^r(y)dy=0,
\]
so again thanks to Poincar\'e inequality and to the compactness of
the embedding of $BV$ in $L^1$ it is relatively compact in the
strong topology of $L^1(B_1(0))$. Finally
\begin{equation}\label{alica3}
\overline R(y)= u^r(y)\Big(\int_0^1 \nabla_z
F\big(x_0+ry,m_r(u)+t(u(x_0+ry)-m_r(u)\big)dt-\nabla_z
F(x_0+ry,m_r(u)\Big)\to 0
\end{equation}
in $L^1(B_1(0))$ thanks to (H2),
proving the claim.

Now choose a sequence $r_h\downarrow 0$ such that
\begin{enumerate}
\item[(a)]$ v^{r_h}\to v^0 $ in $L^1(B_1(0))$,
\item[(b)]$ u^{r_h}\to u^0 \in BV(B_1(0))$ in $L^1(B_1(0))$,
\item[(c)]$F^{r_h} \to F^0 \in BV(B_1(0))$ in $L^1(B_1(0))$,
\item[(d)] the measures $\mu_{x_0,r_h}/\mu(B(x_0,r_h))$ weakly converge in the duality with $C_c(B_1(0))$
to $\nu\neq 0$.
\end{enumerate}
Thanks to \eqref{alica1}, \eqref{alica2} and \eqref{alica3} we have
\[
v^0(y)=F^0(y)+\nabla_z\tilde F(x_0,\tilde u(x_0))u^0(y)
\]
and hence for all $ t\in [0,1] $
\begin{equation}\label{grad}
Dv^0(B_t(0))=DF^0(B_t(0))+\nabla_z\tilde F(x_0,\tilde u(x_0))D
u^0(B_t(0)).
\end{equation}
We now compute the terms $DF^0(B_t(0))$ and $Du^0(B_t(0))$ in the
previous equality. Clearly $DF^{r_h} \weakstarto DF^0$, now choose
$t\in (0,1]$ such that $\nu(B_t(0))\ne 0$ and $\nu(\partial
B_t(0))=0$, then we have (writing in short $D_xF(z)$ for $D_x
F(\cdot,z)$)
\begin{equation}\label{F}
\begin{split}
&DF^0(B_t(0))=\lim_{h \to \infty}  DF^{r_h}(B_t(0))\\
& =\lim_{h\to\infty} \frac{D_x
F(m_{r_h}(u))(B_{tr_h}(x_0))}{\mu(B_{tr_h}(x_0))}
\frac{\mu(B_{tr_h}(x_0))}{\mu(B_{r_h}(x_0))}\\
& = \lim_{h\to\infty} \frac{\big[D_x F(m_{r_h}(u))-D_x F(\tilde
u(x_0))+D_x F(\tilde u(x_0))\big]
(B_{tr_h}(x_0))}{\mu(B_{tr_h}(x_0))}\frac{\mu(B_{tr_h}(x_0))}{\mu(B_{r_h}(x_0))}\\
& = \Big(\lim_{h\to\infty}I(r_h)+\frac{d D_x F(\tilde
u(x_0))}{d\mu}(x_0)\Big)\nu(B_t(0))=\frac{d D_x F(\cdot,\tilde
u(x_0))}{d\mu}(x_0)\nu(B_t(0)),
\end{split}
\end{equation}
since thanks to equation \eqref{ritardo}, (v) and (H3)
\[
\begin{split}
\limsup_{r\downarrow 0} |I(r)|& = \limsup_{r\downarrow 0}
\frac{|D_x F(m_{r}(u))(B_r(x_0))-D_x F(\tilde u(x_0))(B_r(x_0))|}{\mu(B_r(x_0))}\\
& \leq \limsup_{r\downarrow 0}\Bigg(
\frac{|\widetilde{D}_x F(m_{r}(u))(B_r(x_0))-\widetilde{D}_x F(\tilde u(x_0))(B_r(x_0))|}{\mu(B_r(x_0))}+\frac{2\mu(B_r(x_0)\cap \tilde J)}{\mu(B_r(x_0))}\Bigg)\\
& \leq \frac{d \lambda_H}{d\mu}(x_0)\limsup_{r\downarrow
0}\omega_H(|m_r(u)-\tilde u(x_0)|)=0.
\end{split}
\]

A similar (and simpler) calculation gives
\begin{equation}\label{u,v}
Du^0(B_t(0))=\frac{dDu^0}{d\mu}(x_0)\nu(B_t(0)),\qquad
Dv^0(B_t(0))=\frac{dDv^0}{d\mu}(x_0)\nu(B_t(0))
\end{equation}
for all $t\in (0,1]$ such that $\nu(\partial B_t(0))=0$. Comparing
\eqref{F} and \eqref{u,v} with \eqref{grad} we obtain \eqref{mu} and
hence statement (i).

We now prove statement (ii), notice that \eqref{rappDjSigma},
\eqref{deco} and the rectifiability of $B_\sigma$ imply
\[
\begin{split}
D^j v&=D^j v\res J_v=D^j v\res J=D^j v \res (J_u \cup B_\sigma)\\
&=(v^+(x)-v^-(x))\nu_{J_u \cup B_\sigma}\Haus{n-1} \res J_u \cup B_\sigma.
\end{split}
\]
So we have only to check that
\[
v^\pm=F^\pm(x,u^\pm(x))
\]
$\Haus{n-1}$-almost everywhere in $J_u\cup B_\sigma$. To do this,
recall that, thanks to Proposition~\ref{convuniforme}(ii), we have
for $\Haus{n-1}$-almost every $x\in J_u \cup B_\sigma$
\[
\lim_{r \to 0}\mean{B^{\pm}_r(x)}|F(y,z)-F^\pm(x,z)|dy = 0
\]
for every $z \in \R^h $ and that the same is true for $ u $ and $ v $. Choose any such $ x $, then
\[
\begin{split}
v^\pm(x)&=\lim_{r \to 0}\mean{B^\pm_r(x)}F(y,u(y))dy\\
&=\lim_{r \to 0}\mean{B^\pm_r(x)}F(y,u^\pm(x))dy+ \lim_{r \to 0}\mean{B^\pm_r(x)}F(y,u(y))-F(y,u^\pm(x))dy\\
&=F^\pm(x,u^\pm(x)),
\end{split}
\]
since
\[
\lim_{r \to 0} \mean{B^\pm_r(x)}|F(y,u(y))-F(y,u^\pm(x))|dy \le \lim_{r \to 0} M  \mean{B^\pm_r(x)}|u(y)-u^\pm(x)|dy=0.
\]
So, part (b) of the theorem follows with $\mathcal N_F= B_\sigma$.

Finally we prove \eqref{eq:appdif}. Choosing $\mu= \sigma +
M|Du|$ in \eqref{diffuse}, multiplying both sides by $d\mu /d \Leb
n$ and recalling that if $\nu\ll\mu$ then
\[
\frac{d \nu}{d \Leb n}=\frac{d \nu}{d \mu}\frac{d \mu}{d \Leb n},
\]
 we get
\[
\frac {d D v }{d \Leb n} = \frac {D_x F(\cdot,\tilde u(x))}{d\Leb
n}+ \nabla_z\tilde F(x,\tilde u(x))\frac{dDu}{ d\Leb n}.
\]
Using equation \eqref{absolute} and Lemma \ref{lemma ac} we thus
obtain:
\[
D^a v= \nabla v\, d \Leb n= \nabla_x F(x,\tilde u(x))\, d \Leb n +
\nabla_z\tilde F(x,\tilde u(x))\nabla u(x)\, d \Leb n
\]
where $\nabla v$, $\nabla_x F(x,z)$ and $\nabla u$ are the
approximate gradients of the maps $v$, $F(\cdot,z)$ and $u$ (notice
again that the composition $\nabla_x F(x,\tilde u (x))$ is well
defined thanks to the Scorza Dragoni Theorem).
\qed
%


\def\cprime{$'$}
\providecommand{\bysame}{\leavevmode\hbox to3em{\hrulefill}\thinspace}
\def\MR#1{}
\providecommand{\MRhref}[2]{%
  \href{http://www.ams.org/mathscinet-getitem?mr=#1}{#2}
}
\providecommand{\href}[2]{#2}

\end{document}